\documentclass[11pt,notitlepage,a4paper,reqno,
]{amsart}

\usepackage{amssymb}
\usepackage{amstext}
\usepackage{textcomp}
\usepackage{latexsym}
\usepackage{amsfonts}

\newcommand{\medint}{-\kern  -,375cm\int}

\theoremstyle{plain}
\newtheorem{theorem}{Theorem}[section]
\newtheorem{corollary}[theorem]{Corollary}
\newtheorem{lemma}[theorem]{Lemma}
\newtheorem{proposition}[theorem]{Proposition}
\theoremstyle{definition}

\theoremstyle{plain}

\theoremstyle{plain}

\numberwithin{equation}{section} \makeatletter

\renewcommand{\p@enumi}{\thesection.}
\makeatother  \allowdisplaybreaks

\thanks{{\em Acknowledgment:}  The authors are members of  the Gruppo Nazionale per l'Analisi Matematica, 
la Probabilit\`a e le loro Applicazioni (GNAMPA) of the Istituto Nazionale di Alta Matematica (INdAM).
The authors gladly take the opportunity to thank GNAMPA, INdAM, UNIBO, UNIFI, UNIVAQ for the support}

\title[]{Local boundedness for solutions \\ of a class of nonlinear elliptic systems}
\author[G. Cupini - F. Leonetti - E. Mascolo]{Giovanni Cupini - Francesco Leonetti - Elvira Mascolo}
\address{Giovanni Cupini: Dipartimento di Matematica, Universit\`a di
 Bologna\\ Piazza di Porta S.Donato 5,
40126 - Bologna, Italy}
\email{giovanni.cupini@unibo.it}

\address{ Francesco Leonetti:  Dipartimento di Ingegneria e Scienze dell'Informazione  e  Matematica, Universit\`a di
 L'Aquila \\ Via Vetoio snc - Coppito, 67100  - L'Aquila, Italy} 
  \email{francesco.leonetti@univaq.it}
	\address{ Elvira Mascolo:  Dipartimento di Matematica e Informatica ``U. Dini'', Universit\`a di
 Firenze\\ Viale Morgagni 67/A,
50134 - Firenze, Italy} 
 \email{elvira.mascolo@unifi.it}


\begin{document}

\begin{abstract}
In this paper we are concerned with the  regularity of solutions to  a 
nonlinear elliptic system of $m$ equations in divergence form, satisfying $p$ growth from below and $q$ growth from above, with $p \leq q$; this case is known as $p, q$-growth conditions. 
Well known counterexamples, even in the simpler case  $p=q$, show that solutions to systems may be singular; so, it is necessary to add suitable structure conditions on the system that force solutions to be regular. 
Here we obtain 
local boundedness of solutions under a componentwise coercivity condition. 
Our result  is obtained by proving that each component $u^\alpha$ of the solution $u=(u^1,...,u^m)$ 
satisfies  an improved Caccioppoli's inequality and we get the boundedness of  $u^{\alpha}$  by applying De Giorgi's  iteration method, provided the two exponents $p$ and $q$ are not too far apart.
Let us remark that, in dimension $n=3$ and when $p=q$, our result works for $\frac{3}{2} < p < 3$, thus it complements the one of Bjorn 
whose technique 
allowed her to deal with $p \leq 2$ only.
In the final section, we provide 
applications of our result.
\end{abstract}

\maketitle


\section{Introduction}
\label{Introduction}

In this paper we are concerned with the  regularity of solutions to  a 
nonlinear elliptic system of $m$ equations in divergence form 
\begin{equation}\label{dirichlet}
\displaystyle 
\sum_{i=1}^n\frac{\partial }{\partial x_i}\left(A_i^\alpha(x,Du(x))\right)=0, \quad 1\le \alpha\le m,
\end{equation} where 
$x \in \Omega$ and $\Omega$  is a  bounded open set in $\mathbb{R}^n$, $n\ge 2$. 
The function  $u:\Omega \subset \mathbb{R}^n \to \mathbb{R}^m$, 
 has components $(u^1,...,u^m)$; then $Du(x)$ is the $m \times n$ matrix $\left(\displaystyle  \frac{\partial u^\alpha}{\partial x_i}(x) \right)^{\alpha=1,...,m}_{i=1,...,n}$. 

We assume that  $A_i^{\alpha}:\Omega\times \mathbb{R}^{m \times n}\to
\mathbb{R}$,  $1\le i\le n$, $1\le \alpha\le m$, are  Carath\'eodory functions satisfying  
for  every $x\in \Omega$ and for every 
$z=(z^1,\ldots, z^m)^T\in \mathbb{R}^{m \times n}$ the following $p,q$-growth assumptions:
\begin{equation}
\label{(H1)} \nu |z^{\alpha}|^p - a(x)\le \displaystyle \sum_{i=1}^n
A_i^{\alpha}(x,z)z_i^{\alpha}\qquad \forall \alpha\in \{1,\cdots,m\},\end{equation}
\begin{equation}
\label{(H2)}\sum_{i=1}^n|A_i^{\alpha}(x,z)|\le M \left( |z|^{q-1} + b(x) \right),
\end{equation}
where  $1<p\le q$, $p < n$, $\nu,M>0$,   $a\in L_{\rm loc}^{\tau_1}(\Omega)$ and $b\in L_{\rm loc}^{\tau_2}(\Omega)$ are non-negative functions, with $1<\tau_i\le +\infty$, $i=1,2$, and $\tau_2\ge \frac{q}{q-1}$.

\medbreak

Let us recall that $u\in W_{\rm loc}^{1,q}(\Omega;\mathbb{R}^m)$ is a weak solution of \eqref{dirichlet} if
\begin{equation}
\label{weak_solution}
\int_{B} \sum\limits_{\beta = 1}^{m} \sum\limits_{i=1}^{n} A_i^{\beta}(x,Du(x)) D_i \psi^{\beta} (x) \, dx = 0,
\end{equation}
for every open set $B \Subset \Omega$ and for every $\psi \in W_{0}^{1,q}(B;\mathbb{R}^m)$.

 As usual we denote with $p^* = \frac{np}{n-p}$ the Sobolev exponent and the H\"older conjugate exponent $p^\prime = \frac{p}{p-1}$ when $p \in (1, +\infty)$. We use
  the position $\frac{1}{+\infty}=0$.

\medbreak

\noindent 
Our regularity  result is the following.

\begin{theorem} 
Assume that \eqref{(H1)} and \eqref{(H2)} hold, with  $1 < p < n$, $p\le q$ and $1<\tau_1,\tau_2\le +\infty$, 
satisfying 
\begin{equation}
\label{equiv}  
q<p^*\frac{n}{p(n+1)},\qquad  \tau_1>\frac{n}{p}, \qquad \tau_2\ge \frac{q}{q-1}.
\end{equation}
Then any  weak solution $u\in W_{\rm loc}^{1,q}(\Omega;\mathbb{R}^m)$ of \eqref{dirichlet} is locally bounded.
\label{t:boundedness}
\end{theorem}

In the vector-valued case, as suggested by well known counterexamples 
 \cite{deg}, \cite{Mazja},  \cite{giumir}, \cite{frehse1970},  \cite{nec}, \cite{John_Necas_Stara}, \cite{Soucek}, \cite{John_Maly_Stara},   \cite{landes1}, 
\cite{Hao_Leonardi_Steinhauer}, \cite{Hao_Leonardi_Necas},  \cite{sverakyan},  \cite{frehse},   \cite{MonSav},  special structures  on
the operator are required for
everywhere regularity, even under reasonable assumptions on the coefficients; see also the surveys   \cite{Mingione},  
 \cite{Mingione2} and \cite{Kristensen_Mingione}.  
 
In the  literature there are still few contributions about the boundedness of weak solutions to elliptic systems. 
Ladyzhenskaya and Ural'tseva (\cite{Lady-Ural}, Chapter 7) first
proposed the local boundedness of solutions $u=\left(
u^{1},u^{2},\ldots ,u^{m}\right) $ to the \textit{linear} elliptic
system
\begin{equation}\label{introduzione2}\begin{aligned}
\sum_{i=1}^{n}\frac{\partial }{\partial x_{i}} & \left(
\sum_{j=1}^{n}  a_{ij}\left( x\right) \,u_{x_{j}}^{\alpha
}+\sum_{\beta =1}^{m}b_{i}^{\alpha \beta }\left( x\right)
\,u^{\beta }+f_{i}^{\alpha }\left( x\right) \right)
\\ &+\sum_{i=1}^{n}\sum_{\beta =1}^{m}c_{i}^{\alpha \beta }\left(
x\right) \,u_{x_{i}}^{\beta }+\sum_{\beta =1}^{m}d^{\alpha \beta
}\left( x\right) \,u^{\beta }=f^{\alpha }\left( x\right)
\,,\;\;\;\;\;\;\forall \,\alpha =1,2,\ldots ,m,
\end{aligned}\end{equation}%
with bounded measurable coefficients $a_{ij}\,,b_{i}^{\alpha \beta
}\,,c_{i}^{\alpha \beta }\,,d^{\alpha \beta }$ and given functions $%
f_{i}^{\alpha }\,,f^{\alpha }$. Here the \textit{structure condition} is
stated in terms of the positive definite $n\times n$ matrix $\left(
a_{ij}\right) $, which \textit{does not depend} on $\alpha,\beta $. 
 In  \cite{Meier} Meier extended these results to 
 nonlinear elliptic systems of the form
\begin{equation}\label{e:introMeier}\sum_{i=1}^n\frac{\partial }{\partial x_i}\left(A_i^\alpha(x,u,Du)\right)=0,\end{equation}
under the following $p$-growth conditions, $1<p\le n$,
\begin{equation}\label{e:introMeierpbelow}
\sum_{i=1}^n\sum_{\alpha=1}^mA_i^\alpha(x,u,z)z_i^\alpha\ge |z|^{p}-d(x)|u|^{p}-g(x)\end{equation}
\begin{equation}\label{e:introMeierpabove}
|A^\alpha(x,u,z)|\le a|z|^{p-1}+b(x)|u|^{p-1}+e(x)\end{equation}
for $a>0$ and under suitable integrability assumptions on the nonnegative functions $b,e,d,g$. 
Meier introduces
 the so-called {\em indicator function} of the operator 
\begin{equation}\label{e:IA}
{I}_{A}(x, u, Du) := \sum_{\alpha,\beta,i}A_{i}^{\alpha}(x, u, Du)  D_i u^\beta \frac{u^{\alpha} u^{\beta}}{|u|^2}  
\end{equation}
and a pointwise assumption turns out to be crucial in Meier's techniques, indeed a 
weak solution $u$ of \eqref{e:introMeier} is locally bounded if 
\begin{equation}
\label{e:IAge0}
{I}_{A}(x, u, Du)\ge 0
\end{equation}
holds for large values of $|u|$. 
Notice that  \eqref{e:IAge0} is satisfied in 
 linear case \eqref{introduzione2}. 
Assumption \eqref{e:IAge0} is  satisfied also  by some nonlinear operators. For example:
 \begin{equation}
\label{struttura_radiale}
A^\alpha_i (Du) = \sigma(Du) D_i u^\alpha,
\end{equation}
when $0 \leq \sigma$, like in the case of Euler's system of the functional
\begin{equation}
\label{funzionale_ulenbek}
\int F(|Du|) dx,
\end{equation}
where $F$ increases and we take $\sigma(Du) = 
\frac{F^\prime (|Du|)}{|Du|}$. A third example, for which \eqref{e:IAge0} holds true, is given when considering Euler's system of the anisotropic integral
\begin{equation}
\label{funzionale_anisotropo}
\int \sum\limits_{i=1}^{n} g_i(|D_i u|) dx,
\end{equation}
where $g_i$ increases and we take 
\begin{equation}
\label{struttura_anisotropa_radiale}
A^\alpha_i (Du) = \frac{g_i^\prime(|D_i u|)}{|D_i u|} D_i u^\alpha,
\end{equation}
see section 4 in \cite{Leonetti_Mascolo}.
Let us look at another example: we set $m=n$ and we consider the polyconvex integral
\begin{equation}
\label{funzionale_policonvesso}
\int (|Du|^p + h(\det Du)) dx,
\end{equation}
where $h$ is convex, $C^1$, bounded from below. In this case Euler's system gives 
\begin{equation}
\label{struttura_policonvessa}
A^\alpha_i (Du) = p |D u|^{p-2} D_i u^\alpha + h^\prime(\det Du) (\operatorname{Cof} Du)^\alpha_i,
\end{equation}
where $(\operatorname{Cof} Du)^\alpha_i$ is the determinant of the $(n-1) \times (n-1)$ matrix obtained from the $n \times n$ matrix $Du$ by deleting row $\alpha$  and column $i$, with the sign given by $(-1)^{\alpha + i}$. It turns out that 
 \begin{equation}
\label{indicatrice_policonvessa}
I_A(x,u,Du)  = p |D u|^{p-2} \sum\limits_{i=1}^{n} 
\left( 
\sum\limits_{\alpha=1}^{n} \frac{u^\alpha}{|u|} D_i u^\alpha
\right)^2 + h^\prime(\det Du) \det Du
\geq  \inf\limits_{\mathbb{R}}h - h(0),
\end{equation}
then we get \eqref{e:IAge0}, provided $h(0) = \inf\limits_{\mathbb{R}} h$; see 
section \ref{s:example}, later in the present paper; see also \cite{leo-petr2011}. 

The previous examples show that Meier's condition allows us to deal with quite a large class of nonlinear systems.
Boundedness results for weak solutions to nonlinear elliptic systems are proved 
by Kr\"omer \cite{Kromer} under assumptions similar to \eqref{e:IAge0}, see also 
 Landes \cite{landes2}.

Actually Meier's  regularity result  is obtained under a weaker assumption, since 
 $I_A$ can be allowed to be negative, but not too much. 

More precisely, under \eqref{e:introMeier} and \eqref{e:introMeierpbelow}, there exist positive constants $\lambda$ and $L$ such that
\begin{equation}
\label{weak_Meier_condition}
{I}_{A}(x, u, z) := \sum_{\alpha,\beta,i}A_{i}^{\alpha}(x, u, z)  z_i^{\beta} \frac{u^{\alpha} u^{\beta}}{|u|^2}
\geq
- 
\left\{
\delta |z|^p + \left(\frac{1}{\delta}\right)^\lambda [d(x) |u|^p + g(x)]
\right\},
\end{equation}
for every $\delta \in (0,1)$, for all $(x,u,z) \in \Omega \times \mathbb{R}^m \times \mathbb{R}^{m \times n}$, with $|u| > L$.

 Let us observe that the following linear  decoupled system
does not verify \eqref{weak_Meier_condition}, 
see \cite{Leonetti_Petricca_AnnMat2014} and section \ref{s:example}, later in the present paper:
\begin{equation}
\label{struttura_lineare_esclusa}
A^\alpha_i (x, Du) = \sigma^\alpha (x) D_i u^\alpha,
\end{equation}
where $m=2$, 
\begin{equation}
\label{sigma(x)=}
\sigma^1 (x) = 18 + 2\sin(|x|^2)\quad \text{ and } \quad\sigma^2 (x) = 2 + \sin(|x|^2).
\end{equation}

 Now we 
consider  
 another example, see \cite{leo-petr2011}, in which the equations are coupled and Meier's condition \eqref{weak_Meier_condition} is not satisfied: it is Euler's system of 
\begin{equation}
\label{funzionale_prodotto}
\int 
\left[
|Du|^2 + h(D_1 u^1 D_1 u^2)
\right] dx
\end{equation}
where $m=2$, $h$ is convex, $C^1$, bounded from below, so that
\begin{equation}
\label{struttura_prodotto}
A^\alpha_i (Du) =  2 D_i u^\alpha +  h^\prime (D_1 u^1 D_1 u^2) D_1 u^{\hat{\alpha}} \delta_{i 1},
\end{equation}
where 
$$\hat{\alpha} = 2 \,\,\text{ if } \alpha=1\,\, \text{ and }\,\, \hat{\alpha} = 1 \,\,\text{ if } \alpha=2;\,\, \text{ moreover, }\,\, \delta_{i 1}=1 \,\,\text{ if }\,\, i=1\\\,\, \text{ and } \delta_{i 1}=0 \,\,\text{ otherwise.}$$ 
Meier's condition \eqref{weak_Meier_condition} is not satisfied, provided  $h^\prime(0) \leq -8$: for instance, $h(t) = 16\sqrt{1+(t-1)^2}$; see section \ref{s:example} later in the present paper.

Combining coefficients $\sigma^\alpha(x)$ similar to \eqref{sigma(x)=} with the nonlinear part of \eqref{struttura_prodotto}, we are able to build an example with $p$ growth that does not satisfy Meier's condition \eqref{weak_Meier_condition}. Indeed, 
\begin{equation}
\label{struttura_prodotto_p}
A^\alpha_i (Du) = \sigma^\alpha (x) p |Du|^{p-2} D_i u^\alpha +  h^\prime (D_1 u^1 D_1 u^2) D_1 u^{\hat{\alpha}} \delta_{i 1},
\end{equation}
where $2 \leq p$, $h$ is convex, $C^1$, bounded from below; $\hat{\alpha}$ and $\delta_{i 1}$ are defined as before. 
Moreover, $m=2$ and
\begin{equation}
\label{sigma(x)=bis}
\sigma^1 (x) = 48 + 3\sin(|x|^2) \text{ and } \sigma^2 (x) = 2 + \sin(|x|^2).
\end{equation}
Meier's condition \eqref{weak_Meier_condition} is not satisfied, provided  $h^\prime(0) \leq 0$: for instance, $h(t) = \left(1+t^2\right)^{p/4}$; see section \ref{s:example} for the details.

In  \cite{Bjorn} Bjorn  obtained boundedness of solutions $u$ of systems without considering  the indicator function but assuming componentwise coercivity:
\begin{equation}
\label{coercivita_su_ogni_componente}
\nu |z^{\alpha}|^p - a(x) - b(x) |u|^p \leq  \sum\limits_{i=1}^n
A_i^{\alpha}(x,u,z)z_i^{\alpha},\,\, \text{ with }\,\,\nu>0.
\end{equation}

Previous assumption \eqref{coercivita_su_ogni_componente} says that, even if row $\alpha$ of the system contains all the components of $z=Du$, after multiplying this row by component $\alpha$ of $z=Du$, from below we only see the $\alpha$ component of $z=Du$ and none of other components. 

 \eqref{coercivita_su_ogni_componente} is satisfied in system \eqref{struttura_lineare_esclusa}, provided $\sigma^\alpha (x) \geq \nu$ for some positive constant $\nu$.  Furthermore, the structure in \eqref{struttura_radiale} guarantees \eqref{coercivita_su_ogni_componente}, provided $\sigma(Du) \geq \nu |Du|^{p-2}$, for some constants $p \geq 2$ and $\nu >0$. 
Let us mention that 
polyconvex structure \eqref{struttura_policonvessa} enjoys \eqref{coercivita_su_ogni_componente}, provided $p \geq 2$, see section \ref{s:example}.
 Finally, systems in \eqref{struttura_prodotto} and \eqref{struttura_prodotto_p} satisfy \eqref{coercivita_su_ogni_componente}: details are in section \ref{s:example}. 

Let us 
observe that the interesting
 Bjorn's technique allows to deal only with the subquadratic case $1 < p \leq 2$.
When $A^\alpha_i$ does not depend on $u$, in Theorem \ref{t:boundedness}, we are able to deal with the case $p_0 < p < n$, for a suitable $p_0=p_0(n)$; in the three dimensional case $n=3$, $p_0=3/2$, so our result complements the one of Bjorn and we get boundedness of solutions of elliptic systems under componentwise coercivity, see details at the end of this introduction.
 
It is worth pointing out that we study 
system satisfing $p,q$-growth, according to Marcellini \cite{Marcellini1991}. Regularity in this case is obtained when $q$ is not far from $p$, see the survey \cite{Mingione} and, more recently, \cite{MarcelliniDiscrContDinSystems2019}, \cite{Mingione-Radulescu}, \cite{Marcellini2021}; inequality  $p \leq q < p^*\frac{n}{p(n+1)}$ tells us that $q$ cannot be too far from $p$.

 We underline that 
the strategy for proving our vectorial regularity result is De Giorgi's elegant and powerful method, see \cite{degiorgi1}.  Precisely, 
 we prove separately that each component $u^\alpha$ satisfies a suitable Caccioppoli-type inequality, a decay of  the ``excess'' on super-(sub-) level sets of 
$u^\alpha$ that allow to apply iteration arguments and, eventually,  the  local boundedness of the $\alpha$-th component of $u$.  
A similar strategy has been successfully applied in  \cite{CLM-Poli}  to prove the  boundedness of local minimizers of polyconvex functionals satisfying a non-standard growth, 
see also \cite{CGGL}, \cite{CFLM} and \cite{Shan_Gao}.
Local boundedness of weak solutions to some elliptic systems with anisotropic or $p,q$ growth has been proved in \cite{CMM} by using Moser's iteration technique.  
In \cite{leo-petr2011} and \cite{Soft}, a kind of maximum principle has been proved for systems verifying a condition similar to \eqref{coercivita_su_ogni_componente}; see also \cite{PalaSoft}.

We try to explain why we are able to consider values of $p$ larger than the ones considered in \cite{Bjorn}. Bjorn uses Caccioppoli inequality on superlevel sets $\{ v > k \}$ with the {\em same} exponent $p$ both for $Dv$ and $v -k$. We use Caccioppoli inequality on superlevel sets with {\em different} exponents: $p$  for $Dv$ and $p^*$ for $v -k$. When $p$ is close to $n$, then $p^*$ is, by far, larger than $p$, and this helps a lot. Let us also mention that Bjorn  takes $v = \max \{|u^1|,...,|u^m|\}$, where $u=(u^1,...,u^m)$ is the solution of the system; on the contrary, we take $v=u^\alpha$, the component $\alpha$ of $u$.

\medbreak
Let us discuss inequalities $1 < p < n$, $p \leq q < p^*\frac{n}{p(n+1)}$, 
as required in \eqref{equiv} of our Theorem \ref{t:boundedness}.  
We have to solve $p < p^*\frac{n}{p(n+1)}$ when $1<p<n$. This means that $0 < (n+1) p^2 - n(n+1) p + n^2$; when $n=2$ this is satisfied for every $p$; when $n=3$, it is true for $p \neq \frac{3}{2}$; when $n \geq 4$ the inequality is satisfied for $1<p<p_{-}$ or $p_{+} < p <n$, where
\begin{equation}
\label{p-+}
p_{\pm} = \frac{n}{2} 
\left(
1 \pm \sqrt{\frac{n-3}{n+1}}
\right).
\end{equation}
Note that 
\begin{equation}
\label{p-<2}
1 < p_{-} < 2 < p_{+} < n.
\end{equation} 

If we confine ourselves to the case $p=q$, it is possible to make a comparison with Bjorn \cite{Bjorn}. When $n=2$, we recover Bjorn's boundedness result for every $1 < p < n=2$. When $n=3$, Bjorn's result is limited to $1<p \leq 2$ and we complement it, since we are able to deal with $2 < p < n=3$. When $n \geq 4$, Bjorn's result holds true for $1 < p \leq 2$, our result is valid when $p_{+} < p < n$, so it remains open the case $2 < p \leq p_{+}$.



\medbreak
Our paper is organized as follows. In the next section we present the 
proof of Theorem  \ref{t:boundedness}. In
section \ref{s:example} we give details for some of the previous examples.

\section{Proof of Theorem  \ref{t:boundedness}}
\label{s:dimostrazione}

The proof of    Theorem \ref{t:boundedness} is based on the DeGiorgi method, see \cite{degiorgi1}, suitable for dealing with equations.
 Nevertheless  we apply it in the vectorial 
framework, since we can 
apply it to each component $u^{\alpha}$ of a weak solution $u$ separately.

\medskip
\noindent
\subsection*{STEP 1. Caccioppoli inequality}

 The particular growth conditions \eqref{(H1)} and \eqref{(H2)} guarantee a Caccioppoli inequality for any component $u^{\alpha}$ of $u$  on every superlevel set $\{ u^{\alpha} > k \}$.
 

\begin{proposition}\label{p:Caccioppoli} 
Let us consider the system \eqref{dirichlet} and assume that \eqref{(H1)}, \eqref{(H2)} hold.
Let  $u\in W_{\rm loc}^{1,q}(\Omega;\mathbb{R}^m)$ be a weak solution of \eqref{dirichlet}. 
Let $B_{R}(x_0)\Subset\Omega$ with $|B_{R}(x_0)| \leq 1$;
for
 $k\in \mathbb{R}$, $\alpha = 1,...,m$ and $0<\tau \le R$, denote
$$
A^{\alpha}_{k,\tau}:=\{x \in B_{\tau}(x_0)\,:\, u^\alpha(x)>k \}.
$$
If $q\le p^*$ then, there exists $c=c(n,p,\nu,M)>0$
such that, for every $s,t$ with $0<s<t\le R$, for every $k \in \mathbb{R}$ and for every 
$\alpha = 1,...,m$ we have

 \begin{align}\nonumber \int_{A^{\alpha}_{k,{s}}} |Du^\alpha|^p\,dx\le & c\int_{A^\alpha_{k,{t}}}
 \left(\frac{u^{\alpha}-k}{t-s}\right)^{p^*}\,dx
 \\ &
+c
\left\{
\|D u\|_{L^{q}(B_R(x_0))}^{(q-1)(p^*)^{\prime}} +
\|a\|_{L^{\tau_1}(B_R(x_0))}  +
\|b\|_{L^{\tau_2}(B_R(x_0))}^{(p^*)'}
\right\}
|A^\alpha_{k,t}|^{\vartheta},
 \label{Caccioppoli}\end{align}
where \[\vartheta:=\min\left\{1 - \frac{(p^*)^{\prime}}{q'},1- \frac{1}{\tau_1}, 1- \frac{(p^*)'}{\tau_2}\right\}.\] 
We can take $c=\frac{1+M2^{1+p^*}}{\nu}$.
\end{proposition}

\begin{proof}Fix $\alpha\in \{1,\ldots,m\}$. 
 Consider  a
cut-off function
$\eta\in C_0^{1}(B_t(x_0))$  satisfying the following assumptions:
\begin{equation}  \label{eta}
0\le \eta\le 1,\quad \eta \equiv 1\ \text{in $B_{s}(x_0)$,} \quad \text{$|D \eta|\le \frac{2}{t-s}$.}
\end{equation} 
Define the test function  $\psi = (\psi^1,...,\psi^m) \in  W^{1,q}_{0}(B_t(x_0);\mathbb{R}^m)$, where 
$\psi^{\beta} = 0$ if $\beta \neq \alpha$ and   $\psi^{\alpha} = (u^{\alpha}-k)_+ \, \eta$, where $\tau_+=\max\{\tau,0\}$. 
Notice that 
\[\psi^{\alpha}_{x_i}=\chi_{\{u^{\alpha}>k\}}u_{x_i}^{\alpha}\eta
+ 
\eta_{x_i} (u^{\alpha}-k)_+,\]
where $\chi_{E}(x) = 1$ if $x \in E$ and $\chi_{E}(x) = 0$ otherwise; moreover, 
$f_{x_i}=D_if=\frac{\partial f}{\partial x_i}$.

We insert such a $\psi$ into \eqref{weak_solution} and we get
\begin{equation} 
\sum_{i=1}^n\int_{\{u^{\alpha}>k\}}\ A_{i}^{\alpha}(x,Du)u^{\alpha}_{x_i}\eta
\,dx
  = -
	\sum_{i=1}^n\int_{\{u^{\alpha}>k\}} 
  A_{i}^{\alpha}(x,Du)(u^{\alpha}-k)
	\eta_{x_i}\,dx.
\label{e:(5)}
\end{equation}
By \eqref{(H1)} and \eqref{(H2)} 
 \begin{align}\nonumber &\nu \int_{\{u^{\alpha}>k\}}|Du^{\alpha}|^{p} \eta
\,dx\le \int_{\{u^{\alpha}>k, \eta>0\}}a(x)\eta
\,dx
 \\ \nonumber+ & 
M\int_{\{u^{\alpha}>k\}} (u^{\alpha}-k)|Du|^{q-1} 
|D\eta|\,dx
\\ \nonumber+ & 
M
 \int_{\{u^{\alpha}>k\}} (u^{\alpha}-k) b(x)  
|D\eta|\,dx\\ =:& J_1+J_2+J_3.\label{6}\end{align}
It is easy to  estimate $J_1$, indeed, using H\"older inequality 
\begin{equation}
J_1\le \|a\|_{L^{\tau_1}(B_R(x_0))}|A^\alpha_{k,t}|^{1- \frac{1}{\tau_1}}. 
\label{J1}
\end{equation}
In order to estimate $J_2$, we first use Young inequality with exponents $p^*$ and $(p^*)^{\prime}$. 
\[
J_2\leq
M\int_{A^\alpha_{k,t}} (u^{\alpha}-k)^{p^*}
|D\eta|^{p^*} \,dx
+
M\int_{A^\alpha_{k,t}} |Du|^{(q-1)(p^*)^{\prime}} 
\,dx.
\]
Since $q< p^*$ then $(q-1)(p^*)^{\prime} < q$. Therefore we can use H\"older inequality with first 
exponent $\frac{q'}{(p^*)^{\prime}} > 1$ to estimate the last integral,  obtaining 
\begin{equation}\label{e:stimaDu}
M\int_{A^\alpha_{k,t}} |Du|^{(q-1)(p^*)^{\prime}} 
\,dx
\leq
M
\left(
\int_{A^\alpha_{k,t}} |Du|^{q} 
\, dx
\right)^{\frac{(p^*)^{\prime}}{q'}}
\left|
A^\alpha_{k,t}
\right|^{1 - \frac{(p^*)^{\prime}}{q'}}.
\end{equation}
Thus, if 
 we keep in mind that $|D \eta| \leq 2/(t-s)$, then
\begin{equation} 
 J_2
\leq
M
2^{p^*}
\int_{A^\alpha_{k,t}} 
\left(
\frac{u^{\alpha}-k}{t-s}
\right)^{p^*}
 \,dx
+
M
\left(
\int_{A^\alpha_{k,t}} |Du|^{q} 
\, dx
\right)^{\frac{(p^*)^{\prime}}{q'}}
\left|
A^\alpha_{k,t}
\right|^{1 - \frac{(p^*)^{\prime}}{q'}}.
\label{21}
\end{equation}

 In order to estimate $J_3$, we first use Young inequality with exponents $p^*$ and $(p^*)^{\prime}$:
\[
M\int_{\{u^{\alpha}>k\}} (u^{\alpha}-k) b(x) 
|D\eta|\,dx
\leq
M\int_{A^\alpha_{k,t}} (u^{\alpha}-k)^{p^*}
|D\eta|^{p^*} \,dx
+
M\int_{A^\alpha_{k,t}} b^{(p^*)^{\prime}} 
\,dx;
\]
note that  
 $\tau_2 \geq q^{\prime} > (p^*)^{\prime}$; so, we can use H\"older inequality with first 
exponent $\frac{\tau_2}{(p^*)^{\prime}} > 1$ and we get
\[
M\int_{A^\alpha_{k,t}} b(x)^{(p^*)^{\prime}} 
\,dx
\leq
M
\left(
\int_{A^\alpha_{k,t}} b(x)^{\tau_2} 
\, dx
\right)^{\frac{(p^*)^{\prime}}{\tau_2}}
\left|
A^\alpha_{k,t}
\right|^{1 - \frac{(p^*)^{\prime}}{\tau_2}}.
\]
Once again we use that $|D \eta| \leq 2/(t-s)$, then
\begin{equation} 
 J_3
\leq
M
2^{p^*}
\int_{A^\alpha_{k,t}} 
\left(
\frac{u^{\alpha}-k}{t-s}
\right)^{p^*}
 \,dx
+
M
\left(
\int_{B_R (x_0)} b(x)^{\tau_2} 
\, dx
\right)^{\frac{(p^*)^{\prime}}{\tau_2}}
\left|
A^\alpha_{k,t}
\right|^{1 - \frac{(p^*)^{\prime}}{\tau_2}}.
\label{22}
\end{equation}
Collecting \eqref{6}, \eqref{J1}, \eqref{21}, \eqref{22}, we get 
\begin{eqnarray} 
\nu \int_{A^\alpha_{k,t}}|Du^{\alpha}|^{p} \eta
\,dx
\leq
M
2^{1+p^*}
\int_{A^\alpha_{k,t}} 
\left(
\frac{u^{\alpha}-k}{t-s}
\right)^{p^*}
 \,dx
+
\|a\|_{L^{\tau_1}(B_R(x_0))}|A^\alpha_{k,t}|^{1- \frac{1}{\tau_1}}
\nonumber
\\
+M
\left(
\int_{A^\alpha_{k,t}} |Du|^{q} 
\, dx
\right)^{\frac{(p^*)^{\prime}}{q'}}
\left|
A^\alpha_{k,t}
\right|^{1 - \frac{(p^*)^{\prime}}{q'}}
+
M
\left(
\int_{B_R (x_0)} b(x)^{\tau_2} 
\, dx
\right)^{\frac{(p^*)^{\prime}}{\tau_2}}
\left|
A^\alpha_{k,t}
\right|^{1 - \frac{(p^*)^{\prime}}{\tau_2}}
\nonumber
\\
\le  M
2^{1+p^*}
\int_{A^\alpha_{k,t}} 
\left(
\frac{u^{\alpha}-k}{t-s}
\right)^{p^*}
 \,dx
+
\|a\|_{L^{\tau_1}(B_R(x_0))}|A^\alpha_{k,t}|^{1- \frac{1}{\tau_1}}
\nonumber
\\
+
M
\|Du\|_{L^{q}(B_R(x_0))}
^{(q-1)(p^*)^{\prime}}
\left|
A^\alpha_{k,t}
\right|^{1 - \frac{(p^*)^{\prime}}{q'}}+
M
\|b\|_{L^{\tau_2}(B_R(x_0))}^{(p^*)^{\prime}}
\left|
A^\alpha_{k,t}
\right|^{1 - \frac{(p^*)^{\prime}}{\tau_2}}.
\label{23}
\end{eqnarray}
We keep in mind that $\eta = 1$ on $B_{s}(x_0)$ and $|A^\alpha_{k,t}| \leq |B_{R}(x_0)| \leq 1$: 
inequality \eqref{Caccioppoli} follows by taking $c=\frac{1+M2^{1+p^*}}{\nu}$.

\end{proof}

\noindent
\subsection*{STEP 2: Decay of the ``excess'' on superlevel sets.}
\noindent 

 In this step we consider a    
{\em scalar} Sobolev function $v:\Omega\subset \mathbb{R}^n\to \mathbb{R}$, $n\ge 2$. 
%
%
 \medbreak
 
Let us assume  that $\Omega$ is an open set in $\mathbb{R}^n$ and 
$v$ is a {\em scalar}  function $v \in W_{\rm loc}^{1,p}(\Omega;\mathbb{R})$, $p\ge 1$. 
Fix $B_{R_0}(x_0) \Subset \Omega$, with  $R_0<1$ small enough so that 
\begin{equation}|B_{R_0}(x_0)|< 1 \quad \text{and}\quad   
\int_{B_{R_0}}|v|^{p^*}\,dx<  1.\label{e:leonetti}
\end{equation} Here  $p^*=\frac{np}{n-p}$, since  $p<n$.

For every $R\in (0,R_0]$ we define 
the decreasing sequences 
$$
\rho_{h}:=\frac{R}{2}+\frac{R}{2^{h+1}}=\frac{R}{2}\left(1+\frac{1}{2^h}\right), \qquad 
\bar{\rho}_h:=\frac{\rho_{h}+\rho_{h+1}}{2}=\frac{R}{2}\left(1+\frac{3}{4\cdot 2^h}\right).
$$
Fixed   a  positive constant $d \geq 1$,  define the  increasing sequence of positive real numbers 
$$
k_h:= d\left( 1-\frac{1}{2^{h+1}}\right), \,\,h \in \mathbb{N}.
$$
Moreover,    define the sequence $(J_{v,h})$, 
$$
J_{v,h}:= \int _{A_{k_{h},\rho_{h}}}(v-k_h)^{p^*} \, dx,
$$
where $A_{k,\rho} = \{ v>k \} \cap B_\rho$.
The following result holds (see \cite[Proposition 2.4]{CLM-Poli}, \cite{Fusco-Sbordone}, \cite{Moscariello-Nania}).

\begin{proposition}\label{iterazione} 
Let  $v \in W_{\rm loc}^{1,p}(\Omega;\mathbb{R})$, $p\ge 1$. 
Fix $B_{R_0}(x_0) \Subset \Omega$, with  $R_0<1$ small enough such that  \eqref{e:leonetti} holds.
If there exists 
$0\le \vartheta\le 1$ and $c_0>0$ such that for every  $0<s<t\le R_0$ and for every  $k\in \mathbb{R}$ 
\begin{equation}\label{e:Caccioppoli}\int_{A_{k,s}}|Dv|^p\,dx\le c_0
\left\{
\int_{A_{k,t}}\left(\frac{v-k}{t-s}\right)^{p^*}\,dx+|A_{k,t}|^{\vartheta}
\right\},
\end{equation}
then,  for every  $R\in (0,R_0]$,
\[J_{v,h+1}  \le 
 c(\vartheta,R)\left(2^{\frac{p^*p^*}{p}}\right)^h J^{\vartheta\frac{p^*}{p}}_{v,h},
\]
with the positive constant $c$ independent of $h$.
\end{proposition}

\medbreak

\noindent
\subsection*{STEP 3: Iteration and  proof of Theorem \ref{t:boundedness}}

\noindent

We now resume the proof of Theorem \ref{t:boundedness}.  

We need the following  classical result, see e.g. \cite{giusti}.
\begin{lemma}\label{lemma2}
Let $\gamma >0$ and let $(J_h)$ be  a sequence of real positive numbers, such that
\begin{equation}
\label{ipetesi_giusti}
J_{h+1} \leq A\,\lambda^h J_h^{1+\gamma}\qquad \forall h\in \mathbb{N}\cup\{0\},
\end{equation}
\noindent
with $A>0$ and $\lambda>1$. 
If  $J_{0} \leq A^{-\frac{1}{\gamma}}\lambda^{-\frac{1}{\gamma^2}}$, 
then\  $
J_h \le \lambda^{-\frac{h}{\gamma}}J_{0}  
$\  and \  $\lim_{h\to \infty} J_h=0$.
\end{lemma}

Fix 
$B_{R_0}(x_0) \Subset \Omega$, with  $R_0 < 1$ small enough such that  $|B_{R_0}(x_0)|< 1$ and  $\int_{B_{R_0}}|u|^{p^*}\,dx   <  1$.
By Proposition \ref{p:Caccioppoli} we have that $u^\alpha$ satisfies \eqref{Caccioppoli}; i.e.  for every $0<s<t\le R_0$  and every $k\in \mathbb{R}$,
 \begin{align*}\nonumber \int_{A^{\alpha}_{k,{s}}} |Du^\alpha|^p\,dx\le & c\int_{A^\alpha_{k,{t}}}
 \left(\frac{u^{\alpha}-k}{t-s}\right)^{p^*}\,dx
 \\ &
+c
\left\{
\|D u\|_{L^{q}(B_{R_0}(x_0))}^{(q-1)(p^*)^{\prime}} +
\|a\|_{L^{\tau_1}(B_{R_0}(x_0))}  +
\|b\|_{L^{\tau_2}(B_{R_0}(x_0))}^{(p^*)'}
\right\}
|A^\alpha_{k,t}|^{\vartheta},
 \end{align*}
where \[\vartheta:=\min\left\{1-\frac{(p^*)^\prime}{q'},1- \frac{1}{\tau_1}, 1- \frac{(p^*)'}{\tau_2}\right\}\] 
and  $c=\frac{1+M2^{1+p^*}}{\nu}$.

\noindent Therefore the scalar function $u^\alpha$ satisfies \eqref{e:Caccioppoli} of Proposition \ref{iterazione} with 
constant $c_0$ depending on 
\[\|D u\|_{L^{q}(B_{R_0}(x_0))}^{(q-1)(p^*)^{\prime}}, \quad \|a\|_{L^{\tau_1}(B_{R_0}(x_0))}  \quad \text{and}\quad \|b\|_{L^{\tau_2}(B_{R_0}(x_0))}^{(p^*)'}.\]
 Note that these integrals are finite.

 As above, let us   define   
$$
k_h:= d\left( 1-\frac{1}{2^{h+1}}\right), \,\,h \in \mathbb{N}
$$ 
with    $d \geq 1$ ($d$ will be fixed later) and, for every $R\in (0,R_0]$,   define 
$$
\rho_{h}:=\frac{R}{2}+\frac{R}{2^{h+1}}=\frac{R}{2}\left(1+\frac{1}{2^h}\right), \qquad 
\bar{\rho}_h:=\frac{\rho_{h}+\rho_{h+1}}{2}=\frac{R}{2}\left(1+\frac{3}{4\cdot 2^h}\right)
$$
and 
$$
J_{u^\alpha,h}:= \int _{A^\alpha_{k_{h},\rho_{h}}}(u^\alpha-k_h)^{p^*} \, dx.
$$
 Proposition \ref{iterazione}, applied to $u^\alpha$, gives 
\begin{equation}
\label{decadimento_per_J_h}
J_{u^\alpha,h+1}  \le 
 c(\vartheta,R)\left(2^{\frac{p^*p^*}{p}}\right)^h J^{\vartheta\frac{p^*}{p}}_{u^\alpha,h},
\end{equation}
\noindent
with the positive constant $c$ independent of $h$ and, by \eqref{equiv}, with the exponent $\vartheta \frac{p^*}{p}$ greater than $1$. 
Indeed, 
we notice that 
$q < p^* \frac{n}{p(n+1)}$ is equivalent to $\frac{q}{q-1} > \frac{p^*}{p^*-1} \frac{n}{p}$; therefore 
 \eqref{equiv} implies   
\[\frac{p}{p^*}<\min\left\{ 1-\frac{(p^*)^\prime}{q^\prime},1- \frac{1}{\tau_1}, 1- \frac{(p^*)'}{\tau_2}\right\} = \vartheta,\]
so we get $1 < \vartheta \frac{p^*}{p}$.




 Moreover, since \[J_{u^\alpha,0}=\int_{A^\alpha_{\frac{d}{2},R}}\left(u^\alpha-\frac{d}{2}\right)^{p^*}\,dx\to 0 \quad \text{as $d\to +\infty$},\]
 we can choose $d \geq 1$ large enough, so that 
 \[J_{u^\alpha,0} <
 c(\vartheta,R)^{-\frac{1}{\vartheta\frac{p^*}{p}-1}}\left(2^{\frac{p^*p^*}{p}}
 \right)^{-\frac{1}{(\vartheta\frac{p^*}{p}-1)^{2}}}.\] Therefore,  by Lemma \ref{lemma2},    
  $\lim_{h\to +\infty}J_{u^\alpha,h}=0$. Thus,  $u^\alpha\le d$ a.e. 
	in $B_{\frac{R}{2}}(x_0)$. 
We have so proved that $u^\alpha$ is  locally bounded from above. 

To prove that $u^\alpha$ is locally bounded from below, we can observe that 
$\tilde{u}=-u$ is a weak solution for 
\[
\displaystyle 
\sum_{i=1}^n\frac{\partial }{\partial x_i}\left(\tilde{A}_i^\alpha(x,Du(x))\right)=0, \quad 1\le \alpha\le m,
\] where $\tilde{A}(x,z)=-A(x,-z)$. It is easy to check that $\tilde{A}$ satisfies assumptions analogous to \eqref{(H1)} and \eqref{(H2)}. Therefore, by what previously proved, there exists $d'$ such that $\tilde{u}^\alpha=-u^\alpha\le d'$ a.e. in $B_{\frac{R}{2}}(x_0)$. 
 We have so proved that $u^\alpha\in L^{\infty}(B_{\frac{R}{2}}(x_0))$.  Due to the arbitrariness 
of $x_0$ and $R_0$, we get  $u^\alpha\in L^\infty_{\rm loc}(\Omega)$.

\section{Examples}
\label{s:example}

\subsection*{Example 1}

We consider example \eqref{struttura_policonvessa} that we rewrite for the convenience of the reader:
\begin{equation}
\label{riscrivo_struttura_policonvessa}
A_i^\alpha(z) = p |z|^{p-2} z_i^\alpha + h^\prime(\det z) (\operatorname{Cof} z)_i^\alpha,
\end{equation}
where $m=n$, $z \in \mathbb{R}^{n \times n}$, $\det z = \sum_{i=1}^{n} z _i^\alpha (\operatorname{Cof} z)_i^\alpha$; moreover, $h$ is convex, bounded from below and $C^1$. 
Exploiting the convexity of $h$, we get
\begin{equation}
\label{h_convex}
h(0) \geq h(t) + h^\prime(t) (0-t),
\end{equation}
so that
\begin{equation}
\label{h_convex_bis}
 h^\prime(t) t \geq  h(t) - h(0) \geq \inf\limits_{\mathbb{R}} h - h(0).
\end{equation}
Let us compute the indicator function for this choice of $A$: we get
\begin{eqnarray*}
\label{calcoli_indicatrice_policonvesso}
I_A(x,u,z) = \sum\limits_{i,\alpha,\beta} A_i^\alpha(z) z_i^\beta \frac{u^\alpha u^\beta}{|u|^2} =
\sum\limits_{i,\alpha,\beta} p |z|^{p-2} z_i^\alpha z_i^\beta \frac{u^\alpha u^\beta}{|u|^2} 
+
\sum\limits_{i,\alpha,\beta} h^\prime(\det z) (\operatorname{Cof} z)_i^\alpha z_i^\beta \frac{u^\alpha u^\beta}{|u|^2}
\\
=
p |z|^{p-2} \sum\limits_{i} \sum\limits_{\alpha} z_i^\alpha \frac{u^\alpha }{|u|}  \sum\limits_{\beta} z_i^\beta \frac{u^\beta }{|u|}
+
h^\prime(\det z) \sum\limits_{\alpha,\beta} \frac{u^\alpha u^\beta}{|u|^2} \sum\limits_{i} (\operatorname{Cof} z)_i^\alpha z_i^\beta
\\
=
p |z|^{p-2} \sum\limits_{i} \left(\sum\limits_{\alpha} z_i^\alpha \frac{u^\alpha }{|u|} \right)^2
+
h^\prime(\det z) \sum\limits_{\alpha} \frac{u^\alpha u^\alpha}{|u|^2} \sum\limits_{i} (\operatorname{Cof} z)_i^\alpha z_i^\alpha
\\
=
p |z|^{p-2} \sum\limits_{i} \left(\sum\limits_{\alpha} z_i^\alpha \frac{u^\alpha }{|u|} \right)^2
+
h^\prime(\det z) \det z 
\geq \inf\limits_{\mathbb{R}} h - h(0),
\end{eqnarray*}
where we used the property $\sum\limits_{i} (\operatorname{Cof} z)_i^\alpha z_i^\beta = 0$ if $\beta \neq \alpha$. When $h(0) = \inf\limits_{\mathbb{R}} h$, then strong Meier's condition \eqref{e:IAge0} is satisfied; if $h(0) > \inf\limits_{\mathbb{R}} h$, then weak Meier's condition \eqref{weak_Meier_condition} is verified with $\lambda = 1$, $d(x)=0$ and $g(x) = h(0) - \inf\limits_{\mathbb{R}} h$.
Now, let us verify componentwise coercivity \eqref{coercivita_su_ogni_componente}. We have
\begin{eqnarray*}
\label{calcoli_coercivita_policonvesso}
\sum\limits_{i} A_i^\alpha(z) z_i^\alpha  =
\sum\limits_{i} p |z|^{p-2} z_i^\alpha z_i^\alpha 
+
\sum\limits_{i} h^\prime(\det z) (\operatorname{Cof} z)_i^\alpha z_i^\alpha 
\\
=
p |z|^{p-2} |z^\alpha|^2 
+
h^\prime(\det z) \det z
\geq p |z^\alpha|^p + \inf\limits_{\mathbb{R}} h - h(0),
\end{eqnarray*}
provided $p \geq 2$; 
then \eqref{coercivita_su_ogni_componente} is verified with $\nu = p$, $a(x) = h(0) - \inf\limits_{\mathbb{R}} h$ and $b(x) = 0$.

\subsection*{Example 2}

We consider example \eqref{struttura_lineare_esclusa} that we rewrite for the convenience of the reader:
\begin{equation}
\label{riscrivo_struttura_lineare_esclusa}
A_i^\alpha(x,z) = \sigma^\alpha (x) z_i^\alpha,
\end{equation}
where $m=2$, $\sigma^1 (x) = 18 + 2\sin(|x|^2)$ and $\sigma^2 (x) = 2 + \sin(|x|^2)$. Since $\sigma^\alpha (x) \geq 1$, 
it is easy to check \eqref{coercivita_su_ogni_componente}:
\begin{eqnarray*}
\label{calcoli_coercivita_struttura_lineare_esclusa}
\sum\limits_{i} A_i^\alpha(x,z) z_i^\alpha  =
\sum\limits_{i} \sigma^\alpha (x) z_i^\alpha z_i^\alpha 
=
\sigma^\alpha (x) |z^\alpha|^2 
\geq |z^\alpha|^2;
\end{eqnarray*}
so, \eqref{coercivita_su_ogni_componente} is verified with $p=2$, $\nu = 1$, $a(x) = 0$ and $b(x) = 0$.
We are going to show that  \eqref{weak_Meier_condition} is not fulfilled. Indeed, we take $u^1=u^2=s>0$ with $s$ large enough (see \eqref{scelgo_s} later); 
moreover, we take $z_i^\alpha = 0$ if $i \geq 2$, $z_1^1=-s^2$, $z_1^2=2s^2$.
Then $|z|^2 = 5 s^4$, $|u|^2 = 2s^2$, $\frac{u^\alpha u^\beta}{|u|^2} = \frac{1}{2}$ and
\begin{eqnarray*}
\label{calcoli_non_Meier_struttura_lineare_esclusa}
\sum\limits_{i,\alpha,\beta} A_i^\alpha(x,z) z_i^\beta \frac{u^\alpha u^\beta}{|u|^2} =
\frac{1}{2} \sum\limits_{\alpha,\beta} A_1^\alpha(x,z) z_1^\beta = 
\frac{1}{2} \sum\limits_{\alpha} A_1^\alpha(x,z) \sum\limits_{\beta}z_1^\beta 
\\
=
\frac{1}{2} (\sigma^1(x) z_1^1 + \sigma^2(x) z_1^2) (z_1^1 + z_1^2) =
\frac{1}{2} (-\sigma^1(x) + 2 \sigma^2(x)) s^4 = -7s^4  \underbrace{<}_{0<\delta<1} - \delta 7 s^4 
\\
=
 - \delta s^4 \left\{ 5 + 1 + 1 \right\}
\underbrace{\leq}_{(*)} - \delta s^4 \left\{ 5 + \left(\frac{1}{\delta}\right)^{\lambda + 1} \frac{2 d(x)}{s^2} + \left(\frac{1}{\delta}\right)^{\lambda + 1} \frac{ g(x)}{s^4} \right\} 
\\
=
-  \left\{ \delta |z|^2 + \left(\frac{1}{\delta}\right)^{\lambda} \left[ d(x) |u|^2 + g(x) \right] \right\},
\end{eqnarray*}
where (*) is guaranteed by the choice of $s$ as follows
\begin{equation}
\label{scelgo_s}
s = \max 
\left\{ L;
\left[
\left(\frac{1}{\delta}\right)^{\lambda + 1} 2 d(x)\right]^{1/2}; 
\left[ \left(\frac{1}{\delta}\right)^{\lambda + 1}  g(x) \right]^{1/4}
\right\}.
\end{equation}

\subsection*{Example 3}

Let us consider example \eqref{struttura_prodotto} 
that we rewrite for the convenience of the reader:
\begin{equation}
\label{riscrivo_struttura_prodotto}
A^\alpha_i (z) =  2 z_i^\alpha +  h^\prime (z_1^1 z_1^2) z_1^{\hat{\alpha}} \delta_{i 1},
\end{equation}
where $m=2$, $\hat{\alpha} = 2$ if $\alpha=1$ and $\hat{\alpha} = 1$ if $\alpha=2$; moreover, $\delta_{i 1}=1$ if $i=1$ and $\delta_{i 1}=0$ otherwise. Here, $h$ is convex, $C^1$, bounded from below and $h^\prime (0) \leq -8$. For instance, 
\begin{equation}
\label{h=}
h(t) = 16 \sqrt{1 + (t-1)^2}.
\end{equation}
Let us first check \eqref{coercivita_su_ogni_componente}:
\begin{eqnarray*}
\label{calcoli_coercivita_prodotto}
\sum\limits_{i} A_i^\alpha(z) z_i^\alpha  =
\sum\limits_{i} 2 z_i^\alpha z_i^\alpha 
+
\sum\limits_{i} h^\prime(z_1^1 z_1^2) z_1^{\hat{\alpha}} \delta_{i 1} z_i^\alpha 
\\
=
2 |z^\alpha|^2 
+
h^\prime(z_1^1 z_1^2) 
  z_1^{\hat{\alpha}}  z_1^\alpha
=
2 |z^\alpha|^2 
+
h^\prime(z_1^1 z_1^2) z_1^1 z_1^2
\geq 2 |z^\alpha|^2 + \inf\limits_{\mathbb{R}} h - h(0),
\end{eqnarray*}
since $z_1^{\hat{\alpha}}  z_1^\alpha = z_1^1 z_1^2$; 
then \eqref{coercivita_su_ogni_componente} is verified with $\nu = 2$, $p=2$, $a(x) = h(0) - \inf\limits_{\mathbb{R}} h$ and $b(x) = 0$.
We are going to show that  \eqref{weak_Meier_condition} is not fulfilled. Indeed, we take $u^1=u^2=s>0$ with $s$ large enough (see \eqref{scelgo_s} as before); 
moreover, we take $z_1^2 = s^2$ and $z_i^\alpha = 0$ otherwise.
Then $|z|^2 = s^4$, $|u|^2 = 2s^2$, $\frac{u^\alpha u^\beta}{|u|^2} = \frac{1}{2}$ and
\begin{eqnarray*}
\label{calcoli_non_Meier_struttura_prodotto}
\sum\limits_{i,\alpha,\beta} A_i^\alpha(z) z_i^\beta \frac{u^\alpha u^\beta}{|u|^2} = 
\frac{1}{2} \sum\limits_{\alpha,\beta} A_1^\alpha(z) z_1^\beta = 
\frac{1}{2} \sum\limits_{\alpha} A_1^\alpha(z) \sum\limits_{\beta}z_1^\beta 
\\
=
\frac{1}{2} (2 z_1^2 + h^\prime (0) z_1^2) (z_1^2) =  
\frac{1}{2} (2 + h^\prime (0)) s^4 
\underbrace{\leq}_{h^\prime (0) \leq -8} \frac{1}{2} (2 - 8) s^4
=  -3s^4  \underbrace{<}_{0<\delta<1}  - \delta 3 s^4 
\\
=
 - \delta s^4 \left\{ 1 + 1 + 1 \right\}
\underbrace{\leq}_{(*)} - \delta s^4 \left\{ 1 + \left(\frac{1}{\delta}\right)^{\lambda + 1} \frac{2 d(x)}{s^2} + \left(\frac{1}{\delta}\right)^{\lambda + 1} \frac{ g(x)}{s^4} \right\} 
\\
=
-  \left\{ \delta |z|^2 + \left(\frac{1}{\delta}\right)^{\lambda} \left[ d(x) |u|^2 + g(x) \right] \right\},
\end{eqnarray*}
where (*) is guaranteed by the choice of $s$ \eqref{scelgo_s} as before.
In order to show that we can use Theorem \ref{t:boundedness}, we use formula \eqref{h=} 
and we select $n=3$. Then $|h^\prime (t)| \leq 16$ and we get
\begin{equation}
\label{stima_dall_alto_prodotto}
\sum\limits_{i=1}^{3} \left| A_i^\alpha(z) \right| \leq 54 |z|,
\end{equation}
so \eqref{(H2)} is satisfied with $q=2$, $M=54$ and $b(x)=0$, $\tau_2 = +\infty$. Note that previous calculations checked the validity of \eqref{(H1)} with $p=2$, $\nu = 2$, $a(x) = h(0) - \inf\limits_{\mathbb{R}} h = 16(\sqrt{2} - 1)$ and $\tau_1 = +\infty$. Since we selected $n=3$, $q=p=2$, then $\frac{3}{2} = p_0(n) < 2 = q= p < p^* \frac{n}{p(n+1)}$; this implies that \eqref{equiv} is satisfied and we can use our Theorem \ref{t:boundedness} and we get  
the following
\begin{corollary}
If $\Omega$ is a bounded open subset of $\mathbb{R}^3$, then 
all solutions $u \in W^{1,2}_{loc}(\Omega; \mathbb{R}^2)$ of system \eqref{dirichlet}, with $n=3$, $m=2$, \eqref{riscrivo_struttura_prodotto} and \eqref{h=}, are locally bounded in $\Omega$.
\end{corollary}

\subsection*{Example 4}

Let us consider example \eqref{struttura_prodotto_p} that we rewrite for the convenience of the reader: 
\begin{equation}
\label{riscrivo_struttura_prodotto_p}
A^\alpha_i (x,z) =  \sigma^\alpha (x) p |z|^{p-2} z_i^\alpha +  h^\prime (z_1^1 z_1^2) z_1^{\hat{\alpha}} \delta_{i 1},
\end{equation}
where $m=2$, $\sigma^1 (x) = 48 + 3\sin(|x|^2)$ and $\sigma^2 (x) = 2 + \sin(|x|^2)$, $\hat{\alpha} = 2$ if $\alpha=1$ and $\hat{\alpha} = 1$ if $\alpha=2$; moreover, $\delta_{i 1}=1$ if $i=1$ and $\delta_{i 1}=0$ otherwise. Here, $2 \leq p$, $h$ is convex, $C^1$, bounded from below and $h^\prime (0) \leq 0$. For instance, 
\begin{equation}
\label{h=bis}
h(t) =  \left( 1 + t^2 \right)^{p/4}.
\end{equation}
Let us first check \eqref{coercivita_su_ogni_componente}; since $\sigma^\alpha (x) \geq 1$ and $z_1^{\hat{\alpha}}  z_1^\alpha = z_1^1 z_1^2$,
\begin{eqnarray*}
\label{calcoli_coercivita_prodotto_p}
\sum\limits_{i} A_i^\alpha(x,z) z_i^\alpha  =
\sum\limits_{i} \sigma^\alpha (x) p |z|^{p-2} z_i^\alpha z_i^\alpha 
+
\sum\limits_{i} h^\prime(z_1^1 z_1^2) z_1^{\hat{\alpha}} \delta_{i 1} z_i^\alpha 
\\
=
\sigma^\alpha (x) p |z|^{p-2} |z^\alpha|^2 
+
h^\prime(z_1^1 z_1^2) 
  z_1^{\hat{\alpha}}  z_1^\alpha
=
\sigma^\alpha (x) p |z|^{p-2} |z^\alpha|^2 
+
h^\prime(z_1^1 z_1^2) z_1^1 z_1^2
\geq p |z^\alpha|^p + \inf\limits_{\mathbb{R}} h - h(0),
\end{eqnarray*}
then \eqref{coercivita_su_ogni_componente} is verified with $\nu = p$, $a(x) = h(0) - \inf\limits_{\mathbb{R}} h$ and $b(x) = 0$.
We are going to show that  \eqref{weak_Meier_condition} is not fulfilled. Indeed, we take $u^1=u^2=s>0$ with $s$ large enough (see \eqref{scelgo_s_bis} later); 
moreover, we take $z_1^1 = s^2$, $z_1^2 = 0$, $z_2^1 = -2 s^2$, $z_2^2 = 3 s^2$ and $z_i^\alpha = 0$ otherwise.
Then $|z|^2 = 14 s^4$, $|u|^2 = 2s^2$, $\frac{u^\alpha u^\beta}{|u|^2} = \frac{1}{2}$ and
\begin{eqnarray*}
\label{calcoli_non_Meier_struttura_prodotto_p}
\sum\limits_{i,\alpha,\beta} A_i^\alpha(x,z) z_i^\beta \frac{u^\alpha u^\beta}{|u|^2} = 
\frac{1}{2} \sum\limits_{i,\alpha,\beta} A_i^\alpha(x,z) z_i^\beta = 
\frac{1}{2} \sum\limits_{i} \sum\limits_{\alpha} A_i^\alpha(x,z) \sum\limits_{\beta}z_i^\beta 
\\
=
\frac{1}{2} \sum\limits_{i} \sum\limits_{\alpha} \sigma^\alpha (x) p |z|^{p-2} z_i^\alpha \sum\limits_{\beta}z_i^\beta
+
\frac{1}{2} \sum\limits_{i} \sum\limits_{\alpha} h^\prime(z_1^1 z_1^2) z_1^{\hat{\alpha}} \delta_{i 1} \sum\limits_{\beta}z_i^\beta
\\
=
\frac{1}{2} \sum\limits_{i} \sum\limits_{\alpha} \sigma^\alpha (x) p |z|^{p-2} z_i^\alpha \sum\limits_{\beta}z_i^\beta
+
\frac{1}{2}  \sum\limits_{\alpha} h^\prime(z_1^1 z_1^2) z_1^{\hat{\alpha}}  \sum\limits_{\beta}z_1^\beta
\\
=
\frac{p}{2} |z|^{p-2} \sum\limits_{i} 
\left[
\sigma^1 (x) z_i^1 + \sigma^2 (x) z_i^2 
\right] 
\left[
z_i^1 + z_i^2
\right]
+
\frac{1}{2}  
h^\prime(z_1^1 z_1^2)
\left[
z_1^2 + z_1^1
\right]  
\left[
z_1^1 + z_1^2
\right]
\\
=
\frac{p}{2}  |z|^{p-2}  
\left\{
\left[
\sigma^1 (x) z_1^1 + \sigma^2 (x) z_1^2 
\right] 
\left[
z_1^1 + z_1^2
\right]
+  
\left[
\sigma^1 (x) z_2^1 + \sigma^2 (x) z_2^2 
\right] 
\left[
z_2^1 + z_2^2
\right]
\right\}
+
\frac{1}{2}  
h^\prime(z_1^1 z_1^2) 
\left[
z_1^1 + z_1^2
\right]^2
\\
=
\frac{p}{2}  |z|^{p-2}  
\left\{ 
\sigma^1 (x) s^4 
+  
\left[
- 2
\sigma^1 (x)  + 3 \sigma^2 (x)  
\right] s^4
\right\}
+
\frac{1}{2}  
h^\prime(0) s^4
\\
\underbrace{\leq}_{h^\prime (0) \leq 0}
\frac{p}{2}  |z|^{p-2} s^4 
\left\{ -
\sigma^1 (x)  
 + 3 \sigma^2 (x)  
\right\}
= 
\frac{p}{28}  |z|^{p}  
\left\{ -
\sigma^1 (x)  
 + 3 \sigma^2 (x)  
\right\}
= - 3 \frac{p}{2}  |z|^{p}
\\ 
\underbrace{\leq}_{2 \leq p} 
- 3 |z|^{p}
\underbrace{<}_{0<\delta<1}  - \delta |z|^{p} 3  
=
 - \delta |z|^{p} \left\{ 1 + 1 + 1 \right\}
\\
\underbrace{\leq}_{(**)}  - \delta |z|^{p} \left\{ 1 + \left(\frac{1}{\delta}\right)^{\lambda + 1}  d(x) \frac{1}{(\sqrt{7} s)^p} + \left(\frac{1}{\delta}\right)^{\lambda + 1} \frac{ g(x)}{(\sqrt{14} s^2)^p} \right\} 
\\
= 
-  \left\{ \delta |z|^p + \left(\frac{1}{\delta}\right)^{\lambda} \left[ d(x) |u|^p + g(x) \right] \right\},
\end{eqnarray*}
where (**) is guaranteed by the choice \eqref{scelgo_s_bis} of $s$  as follows.
\begin{equation}
\label{scelgo_s_bis}
s= \max 
\left\{
L; \left[ \left(\frac{1}{\delta}\right)^{\lambda + 1}   \frac{d(x)}{(\sqrt{7})^p} \right]^{1/p}; 
\left[\left(\frac{1}{\delta}\right)^{\lambda + 1} \frac{ g(x)}{(\sqrt{14})^p} \right]^{1/(2p)}
\right\}.
\end{equation}
In order to show that we can use Theorem \ref{t:boundedness}, we use formula \eqref{h=bis},  
we select $n=3$ and we require $p<3=n$. Then $|h^\prime (t)| \leq \frac{p}{2} \left( 1 + t^2 \right)^\frac{p-2}{4}$ and we get
\begin{eqnarray*}
\label{stima_dall_alto_prodotto_p}
\sum\limits_{i=1}^{3} \left| A_i^\alpha(x,z) \right| \leq 3 \sigma^\alpha (x) p |z|^{p-1} + 3 \frac{p}{2} \left( 1 + |z|^4 \right)^\frac{p-2}{4} |z|
\leq 153 p |z|^{p-1} + 3 \frac{p}{2} \left( 1 + |z|^4 \right)^{\frac{p-1}{4}}
\\
\leq 153 p |z|^{p-1} + 3 \frac{p}{2} 2^{\frac{p-1}{4}}\left( 1 + |z|^{p-1} \right) 
\leq p( 153 + 2^{1+\frac{p-1}{4}}) (|z|^{p-1} + 1),
\end{eqnarray*}
so \eqref{(H2)} is satisfied with $q=p$, $M=p( 153 + 2^{1+\frac{p-1}{4}})$ and $b(x)=1$, $\tau_2 = +\infty$. Note that previous calculations checked the validity of \eqref{(H1)} with $\nu = p$, $a(x) = h(0) - \inf\limits_{\mathbb{R}} h = 0$ and $\tau_1 = +\infty$. Since we selected $n=3$, $q=p\in [2, 3)$, then $\frac{3}{2} = p_0(n) < 2 \leq q = p < p^* \frac{n}{p(n+1)}$; this implies that \eqref{equiv} is satisfied and we can use our Theorem \ref{t:boundedness} and we get  
the following
\begin{corollary}
If $\Omega$ is a bounded open subset of $\mathbb{R}^3$, then 
all solutions $u \in W^{1,p}_{loc}(\Omega; \mathbb{R}^2)$ of system \eqref{dirichlet}, with $2 \leq p < 3 =n$, $m=2$, \eqref{riscrivo_struttura_prodotto_p} and \eqref{h=bis}, are locally bounded in $\Omega$.
\end{corollary}

\end{document}